\documentclass{amsart}
\usepackage{amsthm,amssymb,amsmath}

\usepackage{setspace}
\theoremstyle{definition}
 \newtheorem{definition}{Definition}[section]

\theoremstyle{plain}

\theoremstyle{plain}
 \newtheorem{theorem}[definition]{Theorem}

\theoremstyle{definition}

\theoremstyle{plain}
 \newtheorem{lemma}[definition]{Lemma}

\theoremstyle{plain}
 \newtheorem{corollary}[definition]{Corollary}

\theoremstyle{remark}

\theoremstyle{definitionnition}

\theoremstyle{plain}

\begin{document}

\title{A Note on the Infinitude of Prime Ideals in Dedekind Domains}
\markright{Infinitude of Prime Ideals in Dedekind Domains}
\author{Jos\'e A. V\'elez-Marulanda}
\address{Department of Mathematics \& Computer Science, Valdosta State University, Valdosta GA 31698}
\email{javelezmarulanda@valdosta.edu}

\maketitle
\begin{abstract}
Let $R$ be an infinite Dedekind domain with at most finitely many units, and let $K$ denote its field of fractions. We prove the following statement. If $L/K$ is a finite Galois extension of fields and $\mathcal{O}$ is the integral closure of $R$ in $L$, then $\mathcal{O}$ contains infinitely many prime ideals. In particular, if $\mathcal{O}$ is further a unique factorization domain, then $\mathcal{O}$ contains infinitely many non-associate prime elements. 
\end{abstract}
\renewcommand{\labelenumi}{\textup{(\roman{enumi})}}
\renewcommand{\labelenumii}{\textup{(\roman{enumi}.\alph{enumii})}}
\numberwithin{equation}{section}
\section{Introduction}
Let $R$ be a integral domain that has at most finitely many units.  We denote by $R^\times$ the set of all units of $R$, i.e., $u\in R^\times$ if and only if there exists $v\in R$ such that $uv=1$. We denote by $R[x]$ the ring of all polynomials in the variable $x$ with coefficients in $R$.  Let $a$ and $b$ be arbitrary elements of $R$. We say that $a$ {\it divides} $b$ in $R$ and write $a|b$, provided that there exits an element $c$ in $R$ such that $b=ac$. We say that $a$ and $b$ are {\it associate} provided that there exists $u$ in $R^\times$ such that $a=ub$. Let $p$ be an element that is neither zero nor a unit in $R$. Recall that  $p$ is said to be {\it prime} if for all elements $r$ and $s$ in $R$ such that $p|rs$ then either $p|r$ or $p|s$. We denote by $K$ the field of fractions of $R$, i.e., $K=\{\frac{p}{q}: p, q\in R \text{ and } q\not=0\}$. Recall that an ideal $\mathfrak{p}$ of $R$ is said to be {\it prime}, provided that for all elements $a,b\in R$, if $ab\in \mathfrak{p}$, then either $a\in \mathfrak{p}$ or $b\in \mathfrak{p}$.
It follows from \cite[Exercise \S 1-1.8]{kaplansky} that $R$ contains infinitely many prime ideals.  If $R$ is further a unique factorization domain, then it follows from \cite[Proposition 2.2 (ii)]{velezcaceres} that $R$ contains infinitely many non-associate prime elements. 
Assume that $R$ is a subring of a ring $L$. Recall that an element $\alpha\in L$ is \textit{integral} over $R$ if there exists a monic polynomial $f\in R[x]$ such that $f(\alpha)=0$. In particular, when 
$R=\mathbb{Z}$, the element $\alpha$ is said to be an \textit{algebraic integer} in $L$. It is a well-known result that the set $B$ consisting of all the elements that are integral over $R$ is a ring, which is called the \textit{integral closure} of $R$ in $L$ (see e.g. \cite[Corollary 5.3]{atiyah}).  In particular,  if $R=\mathbb{Z}$ and $L$ is a field containing $\mathbb{Z}$, the integral closure of $\mathbb{Z}$ in $L$  is called the \textit{ring of integers} of $L$, and we denote this ring by $\mathcal{O}_L$. For example, let $d$ be a square-free integer and consider $\mathbb{Q}(\sqrt{d})=\{a+b\sqrt{d}: a, b\in \mathbb{Q}\}$.  The ring of integers in $\mathbb{Q}(\sqrt{d})$ is 
\begin{equation}\label{w}
\mathcal{O}_{\mathbb{Q}(\sqrt{d})}=\mathbb{Z}[\omega]=\{a+b\omega: a, b\in \mathbb{Z}\},
\end{equation}
where 
\begin{equation*}
\omega=\begin{cases} \sqrt{d}, &\text{ if $d\equiv 2, 3\mod 4$}\\
\frac{1+\sqrt{d}}{2}, &\text{ if $d\equiv 1\mod 4$}\end{cases}
\end{equation*}
We say that  $R$ is \textit{integrally closed} if $R$ is equal to its integral closure in its field of fractions. In particular, $\mathbb{Z}$ is integrally closed.  

For example, let $d\in \{-1,-2,-3,-7,-11,-19,-43,-67,-163\}$ and consider the ring of integers $\mathcal{O}_{\mathbb{Q}(\sqrt{d})}$ of $\mathbb{Q}(\sqrt{d})$. By \cite[Theorem 13.2.5]{artin}, $\mathcal{O}_{\mathbb{Q}(\sqrt{d})}$ is an infinite unique factorization domain with finitely many units, which implies that there are infinitely many non-associative prime elements in $\mathcal{O}_{\mathbb{Q}(\sqrt{d})}$. This gives rise the discussion of the infinitude of prime elements in rings such as $\mathcal{O}_{\mathbb{Q}(\sqrt{2})}=\mathbb{Z}[\sqrt{2}]$. Note that $\mathcal{O}_{\mathbb{Q}(\sqrt{2})}$ is a unique factorization domain that contains infinitely many units, for if $\alpha=1+\sqrt{2}$, then for all integers $n\geq 0$, $\alpha^n$ is a unit in $\mathcal{O}_{\mathbb{Q}(\sqrt{2})}$ (see e.g. \cite[\S 13.9]{artin}). In this note, we prove the infinitude of prime elements in $\mathcal{O}_{\mathbb{Q}(\sqrt{2})}$ by proving the following result.

\begin{theorem}\label{thm1.1}
Assume that $R$ is a Dedekind domain. If $L/K$ is a finite Galois extension of fields and $\mathcal{O}$ is the integral closure of $R$ in $L$, then $\mathcal{O}$ contains infinitely many prime ideals. 
\end{theorem}

Recall that $R$ is a {\it Dedekind domain} if  $R$ is Noetherian, $1$-dimensional and integrally closed. If $(R, K, \mathcal{O}, L)$ as in Theorem \ref{thm1.1}, then $\mathcal{O}$ is also a Dedekind domain (see e.g. \cite[Thm. I.6.2]{lorenzini}), and $R$ satisfies the property of unique factorization of ideals, i.e., every non-trivial ideal $\mathfrak{a}$ of $R$ can be written as $\mathfrak{a}= \mathfrak{p}_1\cdots\mathfrak{p}_s$, where for all $i\in \{1,\ldots, s\}$, $\mathfrak{p}_i$ is a prime ideal of $R$, and this factorization is unique up to the order of the factors (see e.g. \cite[Thm. 6.3]{lorenzini}). 

Since the ring $\mathbb{Z}$ is an infinite Dedekind domain that has finitely many units, and since a Dedekind domain is a unique factorization domain if and only if  it is a principal ideal domain, we get the following immediate consequence of Theorem \ref{thm1.1}.

\begin{corollary}\label{cor1.2}
Assume that $L/\mathbb{Q}$ is a finite Galois extension of fields. Then the ring of integers $\mathcal{O}_L$ contains infinitely many prime ideals. In particular, if $\mathcal{O}_L$ is a unique factorization domain, then $\mathcal{O}_L$ contains infinitely many non-associate prime elements. 
\end{corollary}

It follows from Corollary \ref{cor1.2} that $\mathcal{O}_{\mathbb{Q}(\sqrt{2})}$ contains infinitely many non-associate prime elements.

For further details concerning integral closures and Dedekind domains, see e.g. \cite[Chapter 5]{atiyah}, \cite[\S VIII.5]{hungerford} and \cite[Chapters I, III \& IV]{lorenzini}.

\section{Proof of the main result}

Let $(R,K, \mathcal{O}, L)$ be as in Theorem \ref{thm1.1}. For all ideals $\mathfrak{A}$ of $\mathcal{O}$, define 
\begin{equation}
N_{\mathcal{O}/R}(\mathfrak{A}):= \left(\prod_{\sigma\in \mathrm{Gal}(L/K)}\sigma(\mathfrak{A})\right)\cap R.
\end{equation}
It follows from \cite[Lemma IV.6.4]{lorenzini} that if $\mathfrak{P}$ is a maximal ideal of $\mathcal{O}$, then $N_{\mathcal{O}/R}(\mathfrak{P})=\mathfrak{p}^{f_{\mathfrak{P}/\mathfrak{p}}}$, where $\mathfrak{p}=\mathfrak{P}\cap R$ and $f_{\mathfrak{P}/\mathfrak{p}}=[\mathcal{O}/\mathfrak{P}: R/\mathfrak{p}]$. Moreover, if $\mathfrak{A}=\mathfrak{P}_1^{s_1}\mathfrak{P}_2^{s_2}\cdots\mathfrak{P}_k^{s_k}$ is a product of maximal ideals of $\mathcal{O}$, then 
\begin{equation}\label{eqn2.2}
N_{\mathcal{O}/R}(\mathfrak{A})=\prod_{i=1}^kN_{\mathcal{O}/R}(\mathfrak{P}_i)^{s_i}.
\end{equation}

Consider the inclusion map
\begin{equation}
\iota_{\mathcal{O}/R}:R\to \mathcal{O}.
\end{equation}
It follows that $\iota_{\mathcal{O}/R}$ induces an injective map between the set $\mathcal{I}(R)$ of all ideals of $R$ to the set $\mathcal{I}(\mathcal{O})$, which we also denote by $\iota_{\mathcal{O}/R}$ and which is defined as $\iota_{\mathcal{O}/R}(\mathfrak{a})=\mathfrak{a}\mathcal{O}$ for all ideals $\mathfrak{a}$ of $R$. It follows from \cite[Lemma IV.6.7]{lorenzini} that if $n=|\mathrm{Gal}(L/K)|$, then for all ideals $\mathfrak{a}$ of $R$,
\begin{equation}\label{eqn2.4}
N_{\mathcal{O}/R}(\iota_{\mathcal{O}/R}(\mathfrak{a}))=\mathfrak{a}^n.
\end{equation}

The following well-known result is an exercise in \cite[Exercise \S 1-1.8]{kaplansky} (for its proof, see e.g. \cite[Lemma 2.1]{velezcaceres}).
\begin{lemma}\label{lem2.1}
Assume that $A$ is infinite integral domain with at most finitely many units. Then $A$ has infinitely many maximal ideals. In particular, $A$ has infinitely many prime ideals.
\end{lemma} 

We conventionally assume that for all non-zero ideals $\mathfrak{a}$ of $R$, $\mathfrak{a}^0=R$. 
\begin{proof}[Proof of Theorem \ref{thm1.1}]
Assume that $\mathcal{O}$ contains at most finitely many prime ideals, i.e., assume that $\mathfrak{N}_1,\ldots,\mathfrak{N}_k$ is a complete list of all distinct prime ideals of $\mathcal{O}$. Let $\mathfrak{m}$ be a maximal ideal of $R$. Since $R$ is infinite and $R$ contains at most finitely many units, it follows that $\mathfrak{m}$ is non-zero. Since $\iota_{\mathcal{O}/R}(\mathfrak{m})$ is an ideal of $\mathcal{O}$, it follows that there exist non-negative integers $s_1,\ldots, s_k$ such that 
\begin{equation}\label{eqn1}
\iota_{\mathcal{O}/R}(\mathfrak{m})=\mathfrak{N}_1^{s_1}\cdots\mathfrak{N}_k^{s_k}.
\end{equation}
If $n=|\mathrm{Gal}(L/K)|$, then after applying $N_{\mathcal{O}/R}$ to (\ref{eqn1}), and using (\ref{eqn2.2}) and (\ref{eqn2.4}), we get 
\begin{align*}
\mathfrak{m}^n&= N_{\mathcal{O}/R}(\iota_{\mathcal{O}/R}(\mathfrak{m}))\\
&=N_{\mathcal{O}/R}(\mathfrak{N}_1^{s_1}\cdots \mathfrak{N}_k^{s_k})\\
&= N_{\mathcal{O}/R}(\mathfrak{N}_1)^{s_1}\cdots  N_{\mathcal{O}/R}(\mathfrak{N}_k)^{s_k}\\
&= \mathfrak{n}_1^{s_1f_{\mathfrak{N}_1/\mathfrak{n}_1}}\cdots \mathfrak{n}_k^{s_kf_{\mathfrak{N}_k/\mathfrak{n}_k}},
\end{align*}
where for all $i\in \{1,\ldots,k\}$, $\mathfrak{n}_i=\mathfrak{N}_i\cap R$.
Since $R$ is a Dedekind domain,  and since for all $i\in \{1,\ldots, k\}$, $\mathfrak{n}_i$ is a prime ideal of $R$, it follows that there exist suitable non-negative integers $r_1, \ldots, r_k$ such that 
\begin{equation}
\mathfrak{m}=\mathfrak{n}_1^{r_1}\cdots \mathfrak{n}_k^{r_k}.
\end{equation}
Since $\mathfrak{m}$ is a maximal ideal of $R$, there exists $i_0\in \{1,\ldots, k\}$ such that $r_{i_0}>0$. It follows that 
\begin{equation}
\mathfrak{m}\subseteq \mathfrak{n}_{i_0}^{r_{i_0}}\subseteq \mathfrak{n}_{i_0}.
\end{equation}
It follows that from maximality of $\mathfrak{m}$ that $\mathfrak{m}=\mathfrak{n}_{i_0}$. Therefore, the set $\mathrm{Max}(R)$ consisting of all maximal ideals of $R$ is contained in the set $\{\mathfrak{n}_1,\ldots, \mathfrak{n}_k\}$, which implies that $\mathrm{Max}(R)$ is finite. However, Lemma \ref{lem2.1} implies that $\mathrm{Max}(R)$ is infinite, which leads to a contradiction. Hence $\mathcal{O}$ contains infinitely many prime ideals.
\end{proof}

From the proof of Theorem \ref{thm1.1}, we obtain the following result.

\begin{corollary}
Let $A$ be a Dedekind domain, and let $Q$ be its field of fractions. Assume that $F/Q$ is a finite Galois extension, and let $B$ the integral closure of $A$ in $F$. If $B$ contains at most finitely many prime ideals, then so does $A$.
\end{corollary}

\subsection*{Acknowledgment.}
The author was supported by the Release Time for Research Scholarship of the Office of Academic Affairs at the Valdosta State University. This note is dedicated to the memory of Elizabeth "Lizzie" Lohmar (1993-2014), who was an exceptional student in the B.A. degree in Mathematics program at the Valdosta State University. 

\bibliographystyle{amsplain}
\bibliography{VelezMarulandaBIB}

\providecommand{\bysame}{\leavevmode\hbox to3em{\hrulefill}\thinspace}
\providecommand{\MR}{\relax\ifhmode\unskip\space\fi MR }
\providecommand{\MRhref}[2]{%
  \href{http://www.ams.org/mathscinet-getitem?mr=#1}{#2}
}
\providecommand{\href}[2]{#2}
\begin{thebibliography}{1}

\bibitem{artin}
M.~Artin, \emph{Algebra}, second ed., Prentice Hall, 2011.

\bibitem{atiyah}
M.~F. Atiyah and I.~G. MacDonald, \emph{Introduction to {C}ommutative
  {A}lgebra}, Addison-Wesley, 1969.

\bibitem{velezcaceres}
{L.F.} C{\'a}ceres-Duque and {J.A.} V{\'e}lez-Marulanda, \emph{On the
  infinitude of prime elements}, Rev. Colombiana Mat. \textbf{47} (2013),
  no.~2, 167--179.

\bibitem{hungerford}
T.W. Hungerford, \emph{Algebra}, Graduate Texts in Mathematics 73, Springer,
  1974.

\bibitem{kaplansky}
I.~Kaplansky, \emph{Commutative {R}ings}, Polygonal Publishing House, 1994.

\bibitem{lorenzini}
D.~Lorenzini, \emph{An {I}nvitation to {A}rithmetic {G}eometry}, Graduate
  Studies in Mathematics Volume 9, American Mathematical Society, 1996.

\end{thebibliography}


\vfill\eject

\end{document}